\documentclass[12pt]{article}

\usepackage{graphicx, amssymb, latexsym, amsfonts, amsmath, lscape, amscd,
amsthm, color, epsfig, mathrsfs, tikz, enumerate, todonotes, caption, xcolor, soul}

\definecolor{patriarch}{rgb}{0.5, 0.0, 0.5}
\definecolor{brown(traditional)}{rgb}{0.59, 0.29, 0.0}

\newtheorem{theorem}{Theorem}[section]
\newtheorem{conjecture}[theorem]{Conjecture}

\newtheorem{lemma}[theorem]{Lemma}

\newtheorem{question}{Question}

\newtheorem{observation}[theorem]{Observation}

\theoremstyle{definition}
\newtheorem{remark}[theorem]{Remark}

\usepackage[]{mdframed}

\begin{document}

\title{{\bf A linear algorithm for radio $k$-coloring of powers of paths having small diameters}}
\author{
{\sc Dipayan Chakraborty}$\,^{c, d}$ , {\sc Soumen Nandi}$\,^{b}$, {\sc Sagnik Sen}$\,^{a}$\\  {\sc D K Supraja}$\,^{a, b,\footnote{Corresponding author\\ Email address: dksupraja95@gmail.com}}$\\
\mbox{}\\
{\small $(a)$ Indian Institute of Technology Dharwad, India}\\
{\small $(b)$ Netaji Subhas Open University, India}\\
{\small $(c)$ Université Clermont-Auvergne, CNRS, Mines de Saint-Étienne,}\\ {\small Clermont-Auvergne-INP, LIMOS, 63000 Clermont-Ferrand, France}\\
{\small $(d)$ Department of Mathematics and Applied Mathematics,}\\
{\small University of Johannesburg,
Auckland Park, 2006, South Africa}\\
}

\date{\today}

\maketitle

\begin{abstract}
The radio $k$-chromatic number $rc_k(G)$ of a graph $G$ is the minimum integer $\lambda$
such that there exists a function $\phi: V(G) \to \{0,1,\cdots, \lambda\}$ satisfying 
$|\phi(u)-\phi(v)| \geq k+1 - d(u,v)$, where $d(u,v)$ denotes the distance between $u$ and $v$. 
A considerable amount of attention has been 
given to find the exact values or providing polynomial time algorithms to determine $rc_k(G)$ for several basic graph families such as paths, cycles, trees, and powers of paths, usually for some specific values of $k$.

In this article, we find the exact values of $rc_k(G)$ where $G$ is a power of a path  
with diameter strictly less than $k$.
Our proof readily provides a linear time algorithm for assigning a radio $k$-coloring of $G$. 
Furthermore, our 
proof technique is a potential tool for solving the same problem for other classes of graphs having ``small'' diameters.

\medskip

\end{abstract}

\section{Introduction}

The Channel Assignment Problem (CAP) in wireless networks can be modeled using different graph labeling problems with distance constraints~\cite{hale1980frequency}. One of the prominent types of such distance constrained labelings was defined by Griggs and Yeh~\cite{griggs1992labelling} in 1992 as follows: Given non-negative integers $p_1, p_2, \cdots, p_k$, an $L(p_1, p_2, \cdots, p_k)$-labeling of a graph $G$ is labeling its vertices with non-negative integers such that vertices at distance exactly $i$ receive labels that differ by at least $p_i$.

In 2001, Chartrand, Erwin and Zhang~\cite{chartrand2000radio} initiated a focused study on a special case of $L(p_1, p_2, \cdots, p_k)$-labeling, namely, the radio $k$-coloring of graphs, where $p_i = k+1 - i$ for each $i \in [k]$. In other words, a \textit{$\lambda$-radio $k$-coloring} of a graph $G$ is a function 
$\phi: V(G) \to \{0,1,\cdots, \lambda\}$ satisfying $|\phi(u)-\phi(v)| \geq k+1 - d(u,v)$.

For every $u \in V(G)$, the value $\phi(u)$ is generally referred to as the \textit{color} of $u$ under $\phi$. Usually, we pick $\lambda$ in such a way that it has a preimage under $\phi$, and then, we call $\lambda$ to be the span of $\phi$, denoting it by $span(\phi)$. 
The \textit{radio $k$-chromatic number}\footnote{In the case that $diam(G)=k, k+1$ or $k+2$, the radio $k$-chromatic number is alternatively known as the \textit{radio number} denoted by $rn(G)$, the \textit{radio antipodal number} denoted by $ac(G)$ and the \textit{nearly antipodal number} denoted by $ac'(G)$, respectively.}  $rc_k(G)$ is the minimum $span(\phi)$,
where $\phi$ varies over all radio $k$-colorings of $G$. Note that radio $k$-coloring can also be seen as $L(k, k-1, \cdots, 1)$-labeling. For $k=1$, radio $1$-coloring is merely the proper vertex coloring and for $k=2$, radio $2$-coloring is the well-known $L(2,1)$-coloring. Radio $k$-coloring is a generalized version and is extensively studied for fixed values of $k$ (especially when $k=1,2,3$ and $4$). 

\begin{remark}
     In this model, the frequency separation (that is, the same as the difference between the colors used) varies inversely proportional to the distance, which is a correct requirement for a real-life model. Therefore, radio $k$-coloring can be viewed as a true mathematical model of the CAP, albeit of a very basic form.
\end{remark}

In this article, we focus on the theoretical aspects of radio $k$-coloring. 
All the graphs considered in this article are undirected simple graphs and we refer to the book 
``Introduction to Graph Theory'' by West~\cite{D.B.West} for all standard notations and terminologies used.

\bigskip

\subsection{Context and motivation}

\medskip

\begin{enumerate}[(1)]
    \item The radio $2$-chromatic number is the most well-studied restriction of the parameter, besides the radio $1$-chromatic number, which is equivalent to studying the chromatic number of graphs. A major conjecture proposed by Griggs and Yeh~\cite{griggs1992labelling} on general bound for $rc_2(G)$ is as follows:

\begin{conjecture}~\cite{griggs1992labelling}
For a graph $G$ with maximum degree $\Delta$,
$$rc_2(G)\le \Delta^2.$$
\end{conjecture}
The conjecture has been resolved for all $\Delta \geq 10^{69}$ by Havet, Reed and Sereni~\cite{havet20082}. 

\medskip

\item  For a fixed integer $k\ge 2$, we define the {\sc Radio $k$-Coloring} problem as follows:

\begin{mdframed}
\noindent
\textsc{Radio $k$-Coloring}

\noindent
\textit{Instance:} A graph $G$, $\lambda \ge 4$.

\noindent
\textit{Question:} Is $rc_k(G)\le \lambda$?  
\end{mdframed}

\medskip

\noindent The radio $2$-coloring problem is well-studied from an algorithmic perspective as well. Finding the exact values of $rc_2(G)$ for a general graph is proved to be an NP-complete problem~\cite{griggs1992labelling}. Moreover, it is proven to be NP-complete for restricted classes of graphs such as planar graphs, bipartite graphs, chordal graphs, split graphs~\cite{bodlaender2004approximations} and graphs with treewidth two~\cite{fiala2005distance}. 
On the other hand, polynomial algorithms to find $rc_2(\cdot)$ are known for  trees~\cite{hasunuma2013linear}, paths, cycles and wheels~\cite{griggs1992labelling}.

\medskip

\item Since \textsc{Radio $2$-Coloring} is NP-complete for $\lambda \ge 4$~\cite{fiala2001fixed}, it is quite probable that for $k \ge 3$ as well, the problem \textsc{Radio $k$-Coloring} is NP-complete for $\lambda \ge F(k)$, where $F(k)$ is some function of $k$, possibly equal to $rn(P_k)$. However, in this paper we concentrate only on finding the exact values of the parameter $rc_k(.)$ for a specific graph class called powers of paths (defined later); and leave the complexity results as conjectures for future research in the conclusion.

\item 
Finding the exact value of $rc_k(G)$ for a given graph (usually belonging to a particular graph family) offers a huge number of interesting problems. Unfortunately, due to a lack of general techniques  for solving these problems, not many exact values are known to date. One of the best contributions in this front is the work of Liu and Zhu~\cite{liu2005multilevel} who computed the exact value of $rc_k(G)$ where $G$ is a path or a cycle and $k = diam(G)$.

\begin{theorem}~\cite{liu2005multilevel}
Let $P_n$ be a path on $n$ edges. For every integer $n\ge 3$ and $k=diam(P_n)$, 
$$rc_k(P_n)=
\begin{cases}
\dfrac{n^2+4}{2} &  \text{if $n$ is even},\\
 \dfrac{n^2+1}{2} &  \text{if $n$ is odd}.
\end{cases}$$
\end{theorem}

\medskip

\item  For small paths $P_n$, that is, with $diam(P_n) < k$, Kchikech, Khennoufa and Togni~\cite{kchikech2007linear} had established an exact formula for $rc_k(P_n)$.

 \begin{theorem}~\cite{kchikech2007linear}
Let $P_n$ be a path on $n$ edges. For any $n\ge 2$ and $k> diam(P_n)$,
$$rc_k(P_n)=
\begin{cases}
nk-\frac{n^2-1}{2} & \text{if $n$ is odd,}\\
nk-\frac{n^2}{2}+1 & \text{if $n$ is even.}
\end{cases}$$
\end{theorem}

Furthermore, for the infinite path $P_{\infty}$, the best known lower and upper bounds for the infinite path $P_\infty$ are found in two different bodies of works~\cite{das2017lower, kchikech2007linear}.

\begin{theorem}~\cite{das2017lower, kchikech2007linear}
        
Let $P_\infty$ be the infinite path. Then,
$$\dfrac{k^{2}+k}{2} \le rc_{k}(P_\infty) \le
\left\lfloor\frac{k^2+2k}{2}\right\rfloor.$$
\end{theorem}

Moreover, the upper bound was conjectured to be tight by Kchikech, Khennoufa and Togni~\cite{kchikech2007linear} and is still unresolved. 

In the case of powers of infinite path $P_\infty^m$, we see that $k<diam(P_\infty^m)$. So far, lower and upper bounds for $rc_k(P_\infty^m)$ known have been found in~\cite{vcada2013radio,das2023radio}.

 Furthermore, a number of studies on the parameter $rc_k(P_n)$ depending on how $k$ is related to $diam(P_n)$, or $n$ alternatively, have been done by various authors~\cite{kchikech2007linear,khennoufa2005note,kola2009nearly,chartrand2004radio}. 

\medskip 
 
\item So far as works on powers of paths are concerned, one of the remarkable works we know is an exact formula for the radio number $rn(P_n^2)$ of the square of a path $P_n$ by Liu and Xie~\cite{liu2009radio}. 

\begin{theorem}~\cite{liu2009radio}
Let $P_{n}$ be the path on $n$ edges. For any $n\ge 2$ and $k=diam(P_n^2)$,
$$rc_{k}(P_n^2)=
\begin{cases}
k^2+2 & \text{if $n\equiv 0 \pmod{4}$ and $n\ge 8$,} \\
k^2+1 & \text{otherwise}.
\end{cases}$$
\end{theorem}
Recently, in~\cite{das2023radio}, $rc_k(\cdot)$ for powers of
the infinite path is studied.

 \medskip 

 \item For a detailed overview of the topic, we encourage the reader to consult Chapter 7.4 of the dynamic survey on this topic maintained in the Electronic Journal of Combinatorics by Gallian~\cite{gallian2022dynamic} and the survey by Panigrahi~\cite{panigrahi2009survey}.  
 
 \end{enumerate}

 \begin{remark}
    Notice that, finding the exact value of $rc_k(G)$, or providing a polynomial time algorithm to determine it for so-called well-understood families of graphs, such as paths, cycles, trees, and powers of paths, seems to be difficult and challenging given the amount and quality of attention provided to such problems. 
    Moreover, to date, in most cases, the problems are studied for specific values of $k$ while providing tight bounds or polynomial algorithms. In that context, finding the 
    exact value  and 
    providing a polynomial algorithm to determine 
    $rc_k(\cdot)$ for a relatively complicated family of graphs, such as the powers of paths (with small diameters) for all $k$ seems like a naturally challenging problem. This article focuses on solving it. 
 \end{remark}

\subsection{Our contributions}
Progressing along the same line, in this article, we concentrate on powers of paths having ``small diameters'', that is, $diam(P_n^m) < k$ and compute the exact value of $rc_k(P_n^m)$, where $P_n^m$ denotes the $m$-th power graph of a path $P_n$ on $n$ edges\footnote{A path is often characterized by its number of vertices, and sometimes it is also characterized by its number of edges (for example, in~\cite{bondy1976graph,diestel2012graph}). In this article, we have used the notation $P_n$ for a path having $n$ edges (same notation as in~\cite{bondy1976graph}) as that better suits our calculations).}. In other words, the graph $P_n^m$ is obtained by adding edges between the vertices of $P_n$ that are at most $m$ distance apart. 
 Notice that, the so-obtained graph is, in particular, an interval graph. Let us now state our main theorem.

\begin{theorem}\label{th main}
For all $k > diam(P_n^m)$, we have

$$rc_k(P_n^m)=
\begin{cases}
nk - \frac{n^2-m^2}{2m} &\text{ if } \lceil \frac{n}{m} \rceil \text{ is odd and } m|n,\\
nk - \frac{n^2-s^2}{2m} + 1 &\text{ if } \lceil \frac{n}{m} \rceil \text{ is odd and } m \nmid n,\\
nk - \frac{n^2}{2m} + 1 &\text{ if } \lceil \frac{n}{m}\rceil \text{ is even and } m|n,\\
nk - \frac{n^2-(m-s)^2}{2m} + 1 &\text{ if } \lceil \frac{n}{m} \rceil \text{ is even and } m \nmid n,
\end{cases}
$$
where $s \equiv n \pmod m$.  
\end{theorem}

In this article, we develop a robust graph theoretic tool for the proof. Even though the tool is specifically used to prove our result, it can be adapted to prove bounds for other classes of graphs. Thus, we would like to remark that, the main contribution of this work is not only in proving an important result that captures a significant number of problems with a unified proof, but also in devising a proof technique that has the potential of becoming a standard technique to attack similar problems. 
We will prove the theorem in the next section.

Moreover, our proof of the upper bound is by giving a prescribed radio $k$-coloring of the concerned graph, and then proving its validity, while the lower bound proof establishes its optimality. Therefore, as a corollary to Theorem~\ref{th main}, we can say that our proof provides a linear time algorithm to radio $k$-color powers of paths, optimally. 

\begin{theorem}\label{th algo}
For all $k > diam(P_n^m)$, one can provide an optimal radio $k$-coloring of the graph 
$P_n^m$ in $O(n)$ time. 
\end{theorem}

\subsection{Organization of the paper}
We begin Section~\ref{thrm_proofs} with two naming conventions which are used in the lower and upper bound proofs of Theorem~\ref{th main}. We also present a few lemmas which play a significant role in proving the lower and upper bounds. Further, we prove Theorems~\ref{th main} and~\ref{th algo}. Finally, in Section~\ref{conclusion}, we conclude by stating a few interesting open problems.

A preliminary version of this work (with less detailed proofs) appeared in the proceedings of the IWOCA 2023 conference~\cite{chakraborty2023linear}.

\section{Proofs of Theorems~\ref{th main} and~\ref{th algo}}\label{thrm_proofs}
This section is entirely dedicated to the proofs of Theorems~\ref{th main} and~\ref{th algo}. The proof uses specific notations and terminologies developed to make it easier for the reader to follow. The proof is contained in several observations and lemmas and uses a modified and improved version of the DGNS formula~\cite{das2017lower} applicable for graphs having small diameters, that is, less than or equal to $k$.

As seen from the theorem statement, the graph $P_n^{m}$ that we work on is the $m^{th}$ power of the path
 on $(n+1)$ vertices. One crucial aspect of this proof is the naming of the vertices of $P_n^{m}$. In fact, for convenience, we shall assign \emph{two} names to each of the vertices of the graph and use them 
as required, depending on the context. Such a naming convention will depend on the parity of the diameter of $P_n^m$.

\begin{observation}\label{obs diam}
The diameter of the graph $P_n^{m}$ is $diam(P_n^{m})= \lceil \frac{n}{m}  \rceil$. 
\end{observation}

For the rest of this section, we shall fix the notation that $q = \lfloor \frac{diam(P_n^{m})}{2}  \rfloor$.  

\subsection{The naming conventions}
We are now ready to present the first naming convention for the vertices of $P_n^{m}$. For convenience, let us suppose that the vertices of $P_n^{m}$ are placed (embedded) on the $X$-axis having co-ordinates 
$(i,0)$ where $i \in \{0,1, \cdots, n\}$ and two (distinct) vertices are adjacent if and only if their 
Euclidean  distance is at most $m$. 

We start by selecting  the layer $L_0$ consisting of the vertex, named $c_{0}$, say, positioned at $(qm,0)$ for even values of $diam(P_n^{m})$. 
On the other hand, for odd values of $diam(P_n^{m})$, the layer $L_0$ consists of the vertices $c_{0}, c_{1}, \cdots, c_{m}$, say, positioned at $(qm,0), (qm+1,0), \cdots, (qm+m,0)$, respectively, and inducing a maximal clique of size $(m+1)$.
The vertices of $L_0$ are called the \textit{central vertices}, and those positioned to the left and the right side of  the central vertices are naturally called the \textit{left vertices} and the \textit{right vertices}, respectively.

After this, we define the layer $L_i$ as the set of vertices that are at a distance $i$ from $L_0$. Observe that the layer $L_i$ is non-empty for all $i \in \{0, 1, \cdots, q\}$. 
Moreover, notice that, for all $i \in \{1,2, \cdots, q\}$, $L_i$ consists of both left and right vertices. 
In particular, for $i \geq 1$, the left vertices of $L_i$ are named $l_{i1}, l_{i2}, \cdots, l_{im}$, sorted according to the increasing order of their Euclidean  distances from $L_{0}$. 
Similarly, for $i \in \{1,2, \cdots, q-1\}$,
the right vertices of $L_i$ are named $r_{i1}, r_{i2}, \cdots, r_{im}$, sorted according to the increasing order of their Euclidean  distance from $L_{0}$. However, the right vertices of $L_q$ are $r_{q1}, r_{q2}, \cdots, r_{qs}$, where $s=(n+1)-(2q-1)m-|L_0|$ (observe that this $s$ is the same as the $s$ mentioned in the statement of Theorem~\ref{th main}), again sorted according to the increasing order of their Euclidean distances from $L_0$. That is, if $m \nmid n$, then there are $s=(n+1)-(2q-1)m-|L_0|$ right vertices in $L_q$. Besides, every layer $L_i$, for $i \in \{1, 2, \cdots , q-1\}$,  
has exactly $m$ left vertices and $m$ right vertices. This completes our first naming convention.

Now, we move to the second naming convention. This depends on yet another observation. 

\begin{observation}\label{obs distinct}
Let $\phi$ be a radio $k$-coloring of $P_n^{m}$. Then $\phi(x) \neq \phi(y)$ for all distinct $x,y \in V(P_n^{m})$. 
\end{observation}

\begin{proof}
As $diam(P_n^{m}) < k$, the distance between any two vertices of $P_n^{m}$ is at most $k-1$. Thus, their colors must differ by a value of at least $1$. 
\end{proof}

Let $\phi$ be a radio $k$-coloring of $P_n^{m}$. 
Thus, due to Observation~\ref{obs distinct}, it is possible to sort the vertices of $P_n^{m}$ 
according to the increasing order of their colors. 
That is, our second naming convention which names the vertices of $P_n^{m}$ as $v_0, v_1, \cdots, v_n$ satisfying $\phi(v_0) < \phi(v_1) < \cdots < \phi(v_n)$. 
Clearly, the second naming convention depends only on the coloring $\phi$, which, for the rest of this section, will play the role of any arbitrary radio $k$-coloring of $P_n^{m}$.

\subsection{The lower bound}
Next, we shall proceed to establish the lower bound of Theorem~\ref{th main} by showing it to be a lower bound of $span(\phi)$. To do so, however, we need to introduce yet another notation. 
Let $f: V(P_n^{m}) \to \{0,1, \cdots, q\}$ be the function which indicates the layer of a vertex, that is, 
$f(x) = i$ if $x \in L_i$. With this notation, we initiate the lower bound proof with the following result.

\begin{lemma}\label{lem layering color gap}
For any $i \in \{0,1, \cdots, n-1\}$, we have 
\begin{equation*}
\phi(v_{i+1})-\phi(v_{i}) \geq 
k - f(v_i)- f(v_{i+1}) +\epsilon,
\end{equation*}
where $\epsilon=1$ if $diam(P_n^m)$ is even and $\epsilon=0$ if $diam(P_n^m)$ is odd.
\end{lemma}

\begin{proof}
If $diam(P_n^{m})$ is even, then $L_0$ consists of the single vertex $c_0$. 
Observe that, as $v_i$ is in $L_{f(v_i)}$, 
it is at a distance $f(v_i)$ from $c_0$. Similarly, $v_{i+1}$ is at a distance $f(v_{i+1})$ from $c_0$. 
Hence, by the triangle inequality, we have 
$$d(v_i, v_{i+1}) \leq d(v_i,c_0) + d(c_0, v_{i+1}) = f(v_i)+f(v_{i+1}).$$ Therefore, by the definition of radio $k$-coloring, $$\phi(v_{i+1})-\phi(v_{i}) \geq k - f(v_i)- f(v_{i+1}) +1.$$

If $diam(P_n^{m})$ is odd, then $L_0$ is a clique. 
Thus, by the definition of layers and the function $f$, 
there exist vertices $c_j$ and $c_{j'}$ in $L_0$ satisfying $d(v_i,c_j) = f(v_i)$ and 
$d(v_{i+1}, c_{j'}) = f(v_{i+1})$. 
Hence, by triangle inequality again, we have 
$$d(v_i, v_{i+1}) \leq d(v_i,c_j) + d(c_j, c_{j'}) + d(c_{j'}, v_{i+1}) = f(v_i)+1+f(v_{i+1}).$$ Therefore, by the definition of radio $k$-coloring, $$\phi(v_{i+1})-\phi(v_{i}) \geq k - f(v_i)- f(v_{i+1}).$$

Hence we are done. 
\end{proof}

Notice that it is not possible to improve the lower bound of the inequality 
presented in Lemma~\ref{lem layering color gap}. 
Motivated by this fact, 
whenever we have 
\begin{equation*}
\phi(v_{i+1})-\phi(v_{i}) = 
k - f(v_i)- f(v_{i+1}) +\epsilon 
\end{equation*}
for some $i \in \{0,1, \cdots, n-1\}$, we say that the pair $(v_i, v_{i+1})$ is \textit{optimally colored}  by $\phi$. 
Moreover, we can naturally extend this definition to a sequence of vertices of the type 
$(v_i, v_{i+1}, \cdots, v_{i+i'})$ by calling it an \textit{optimally colored sequence} by $\phi$ 
if $(v_{i+j},v_{i+j+1})$ is optimally colored by $\phi$ for all $j \in \{0,1,\cdots, i'-1\}$. 
Furthermore, a \textit{loosely colored sequence} $(v_i, v_{i+1}, v_{i+2}, \cdots, v_{i+i'})$ is a sequence 
that does not contain any optimally colored sequence as a subsequence.

An important thing to notice is that the sequence of vertices $(v_0,  v_1,  \cdots, v_n)$ can be 
written as a concatenation of maximal optimally colored sequences and loosely colored sequences. That is, 
it is possible to write 
 $$(v_0,  v_1,  \cdots, v_n) = Y_0 X_1 Y_1 X_2 \cdots X_t Y_t$$
 where $Y_i$s are loosely colored sequences and $X_j$s are maximal optimally colored sequences. Here, we allow the $Y_i$s to be empty sequences as well. 
 In fact, for $1\le i\le t-1$, a $Y_i$ is empty if and only if there exist two consecutive vertices $v_{s'}$ and $v_{s'+1}$ of $P_n^m$ in the second naming convention such that $(v_{s'}, v_{s'+1})$ is loosely colored and that $X_i = (v_s, v_{s+1}, \cdots, v_{s'})$ and 
$X_{i+1}=(v_{s'+1}, v_{s'+2}, \cdots, v_{s''})$ for some $s \leq s' <s''$. Moreover, $Y_0$ (resp. $Y_t$) is empty if and only if the pair $(v_0, v_1)$ (resp. $(v_{n-1}, v_n)$) is optimally colored.
By convention, empty sequences are always loosely colored and a sequence having a singleton vertex is always optimally colored. From now onward, whenever we mention a radio $k$-coloring $\phi$ of $P_n^{m}$, we shall also suppose an associated concatenated sequence using the same notation as mentioned above. 

Let us now prove a result which plays an instrumental role in the proof of the lower bound.

\begin{lemma}\label{lem tighter counting}
Let $\phi$ be a radio-$k$ coloring of $P_n^{m}$ such that 
$$(v_0, v_1, \cdots, v_n)= Y_0 X_1 Y_1 X_2 \cdots X_t Y_t .$$
Then, we have 
$$span(\phi) \geq \left[n(k+\epsilon) - 2 \sum_{i=1}^{q} i |L_i| \right] + \left[f(v_0) + f(v_n) + \sum_{i=0}^{t} |Y_i| + t -1\right],$$
where $|Y_i|$ denotes the length of the sequence $Y_i$, and $\epsilon=1$ for even values of $diam(P_n^m)$ and $\epsilon=0$ for odd values of $diam(P_n^m)$. 
\end{lemma}

\begin{proof}
We know that $span(\phi) = \phi(v_n) - \phi(v_0)$. 
However, we can expand this  difference as
\begin{align*}
span(\phi) &= \phi(v_n) - \phi(v_0) \\ 
           &= (\phi(v_n) - \phi(v_{n-1})) + (\phi(v_{n-1}) - \phi(v_{n-2}))  + \cdots + (\phi(v_1) - \phi(v_{0}))\\
           &= \sum_{i=0}^{n-1} [\phi(v_{i+1}) - \phi(v_i)].
\end{align*}

\noindent
Notice that, by Lemma~\ref{lem layering color gap}, we have 
$$\phi(v_{i+1})-\phi(v_{i}) \geq k - f(v_i)- f(v_{i+1}) + \epsilon $$
 and, if $(v_i,v_{i+1})$ is loosely colored, 
 then  
$$ \phi(v_{i+1})-\phi(v_{i}) > k - f(v_i)- f(v_{i+1}) + \epsilon.$$

Therefore, if 
$$S= \{ v_i: (v_i,v_{i+1}) \text{ is loosely colored, where } 0 \leq i \leq n-1 \},$$
then we have,
\begin{align*}
span(\phi) &= \sum_{i=0}^{n-1} [\phi(v_{i+1}) - \phi(v_i)] \\ 
           &\geq \lvert S \rvert + \sum_{i=0}^{n-1} [k - f(v_i)- f(v_{i+1})+ \epsilon]\\ 
           &= \lvert S \rvert + n(k +\epsilon) - 
           \sum_{i=0}^{n-1} f(v_i)- \sum_{i=0}^{n-1} f(v_{i+1})\\  
           &= \lvert S \rvert + n(k +\epsilon) 
+f(v_0) + f(v_n) -  2\sum_{i=0}^{n} f(v_i)\\   
&= \lvert S \rvert + n(k +\epsilon) +f(v_0) + f(v_n) 
-  2\sum_{i=0}^{q} i|L_i|.
           \end{align*}

Notice that, to count $|S|$ it is enough to count the lengths of the loosely colored sequences, i.e. the $|Y_i|$s, and the number of transitions between the loosely colored and the optimally colored sequences, i.e. between a $Y_i$ and an $X_i$. 
To be precise, we can write 
         \begin{align*}
\lvert S \rvert &= |Y_0| + (|Y_1|+1)+(|Y_2|+1) + \cdots + (|Y_{t-1}|+1) + |Y_t|\\
&= (t-1) + \sum_{i=0}^{t} |Y_t|.
           \end{align*}

Combining the above two equations therefore, we obtain the result. 
\end{proof}

As we shall calculate the two additive components of Lemma \ref{lem tighter counting} separately, we introduce short-hand notations for them for the convenience of reference. So, let
\begin{equation*}
\alpha_1 = 
n(k+\epsilon) - 2 \sum_{i=1}^{q} i |L_i| \end{equation*}
and 
$$\alpha_2(\phi) = f(v_0) + f(v_n) + \sum_{i=0}^{t} |Y_i| + t - 1.$$
Observe that  $\alpha_1$ and $\alpha_2$ are functions of a number of variables and factors 
such as, $n, m, k, \phi, $ etc. However, to avoid clumsy and lengthy formulations, 
we have avoided writing $\alpha_1$ and $\alpha_2$ as multivariate functions, as their definitions are not ambiguous in the current context. Furthermore, as $k$ and $P_n^{m}$ are assumed to be fixed in the current context 
and, as $\alpha_1$ does not depend on $\phi$ (follows from its definition), it is treated and expressed as a constant as a whole. On the other hand, $\alpha_2$ is expressed as a function of $\phi$. 

Now we shall establish lower bounds for $\alpha_1$ and $\alpha_2(\phi)$ separately to prove the lower bound of 
Theorem~\ref{th main}. Let us start with $\alpha_1$ first.

\begin{lemma}\label{lem lower bound alpha1}
We have 
$$\alpha_1=
\begin{cases}
nk - \frac{n^2+m^2-s^2}{2m} &\text{ if } diam(P_n^{m}) \text{ is even,}\\
nk - \frac{n^2-s^2}{2m} &\text{ if } diam(P_n^{m}) \text{ is odd,}
\end{cases}
$$
where $s=(n+1)-(2q-1)m-|L_0|$.  
\end{lemma}

\begin{proof}
Notice that $|L_i|=2m$ for all $i \in \{1, 2, \cdots, q-1\}$ and $|L_q|=m+s$. So, simply replacing these values in the definition of $\alpha_1$ and using the relation 
$s=n-(2q-1+\epsilon)m$, where $\epsilon=0$ for even values of $diam(P_n^{m})$ and 
$\epsilon=1$ for odd values of $diam(P_n^{m})$, gives us the result. 
\end{proof}

Next, we focus on $\alpha_2(\phi)$. We shall handle the cases with odd
$diam(P_n^{m})$ first.

\begin{lemma}\label{lem lower bound alpha2 odd diam}
We have 
$$\alpha_2(\phi) \geq 
\begin{cases}
0 &\text{ if } diam(P_n^{m}) \text{ is odd and } m|n,\\
1 &\text{ if } diam(P_n^{m}) \text{ is odd and } m \nmid n.
\end{cases}
$$  
\end{lemma}

\begin{proof}
First of all, notice that there is nothing to prove when $ diam(P_n^{m})$ is odd and $m|n$ as $\alpha_2(\phi)$ is always non-negative by definition. However, when $ diam(P_n^{m})$ is odd and $m \nmid n$, it is enough to show that $\alpha_2(\phi) \neq 0$. Suppose the contrary, that is, $\alpha_2(\phi) = 0$. Then, we must have both $f(v_0)=f(v_n)=0$ and $(v_0, v_1, \cdots, v_n) = Y_0X_1Y_1$
having $Y_0=Y_1= \emptyset $. That is, both $v_0$ and $v_n$ must be from $L_0$ and the whole sequence 
$(v_0, v_1, \cdots, v_n)$ must be an optimally colored sequence.

Observe that if $l_{i1}$, for any $i \in \{1,2, \cdots, q\}$, is an element of an optimally colored pair, then the other element must be either $c_m$ or $r_{jm}$ for some $j \in \{1,2,\cdots,q-1\}$. This follows 
from the distance constraints and the definition of an optimally colored pair of vertices. On the other hand, a pair of vertices in which one is $c_m$ and the other is a right vertex is not an optimally colored pair of vertices. Moreover, any pair of left vertices $(l_{ia},l_{i'a'})$ or any pair of right vertices $(r_{jb},r_{j'b'})$ are also loosely colored each.

Thus, $X_1$ must contain a contiguous subsequence of the form $(a_1, b_1, a_2, b_2,\\ \cdots,  a_q, b_q)$ 
where $a_i$s (resp., $b_j$s) are from  
$\{l_{11}, l_{21}, \cdots, l_{q1}\}$ and  
$b_j$s (resp., $a_i$s)  are from 
$\{c_m, r_{1m}, r_{2m},$ 
$\cdots, r_{(q-1)m}\}$.

If $a_1 \in \{l_{11}, l_{21}, \cdots, l_{q1}\}$, then $a_1 \neq v_0$, as $f(v_0)=0 \neq f(a_1)$. 
Thus $a_1 = v_i$ for some $i \geq 1$. This is not possible as $v_{i-1}$ cannot be from 
the set $\{c_m, r_{1m}, r_{2m},$ 
$\cdots, r_{(q-1)m}\}$ and therefore, the pair $(v_{i-1}, v_i)$ is not optimally colored, a contradiction. Hence, $\alpha_2(\phi) \neq 0$.

Similarly, we can arrive at a contradiction if $b_q \in \{l_{11}, l_{21}, \cdots, l_{q1}\}$ and so, $\alpha_2(\phi) \neq 0$ in this case as well. Hence, we are done. 
\end{proof}

Next, we consider the cases with even $diam(P_n^{m})$. Before starting with it though, we are going to introduce some terminologies 
to be used during the proofs. So, let $X_i$ be an optimally colored sequence. As $X_i$ cannot have two consecutive left (resp., right) vertices as elements, the number of left vertices can be at most one more than the number of right vertices and the central vertex, the latter two combined together. Based on this observation, if the number of left vertices is more, equal, or less than the number of right vertices and the central vertex combined in $X_i$, then $X_i$ is called a \textit{leftist}, \textit{balanced}, or \textit{rightist} sequence, respectively.

\begin{lemma}\label{lem lower bound alpha2 even diam}
We have 
$$\alpha_2(\phi) \geq 
\begin{cases}
1 &\text{ if } diam(P_n^{m}) \text{ is even and } m|n,\\
m-s+1 &\text{ if } diam(P_n^{m}) \text{ is even and } m \nmid n,
\end{cases}
$$
where $s \equiv n \pmod m$.  
\end{lemma}

\begin{proof}
For even values of $diam(P_n^{m})$, $L_0$ consists of only the vertex $c_0$. Therefore, at most one of $v_0$ and $v_n$ can be equal to $c_0$ implying $f(v_0)+f(v_n) \geq 1$. This proves the case when $m | n$. So, let us now focus on the case when $m \nmid n$.

We know that there are exactly $(q-1)m+s$ right vertices and one central vertex $c_0$.
Suppose that at most $(q-1)m+s$ vertices among the set of right and central vertices are part of optimally colored sequences of $(v_0, v_1, \cdots, v_n)$.
Thus, the total number of vertices across the $t$ optimally colored sequences will be
$$\sum_{i=1}^{t} |X_i| \leq 2(q-1)m +2s +t.$$ 
That leaves us with 
$$\sum_{i=0}^{t} |Y_i| \geq [(2q-1)m+s+1] - [2(q-1)m +2s +t] = m-s+1-t.$$
Recall that $f(v_0)+f(v_n)\geq 1$. Hence, 
$$\alpha_2(\phi) = f(v_0)+f(v_n) + \sum_{i=0}^{t} |Y_i| +t -1 \geq m-s+1.$$

Therefore, we are left with the case when all $(q-1)m+s+1$ right and central vertices are part of optimally colored sequences of $(v_0, v_1, \cdots, v_n)$.
Suppose that the number of leftist, balanced, and rightist sequences are 
$t_1$, $t_2$, and $t_3$, respectively, where $t_1+t_2+t_3=t$.
In this case
$$\sum_{i=1}^{t} |X_i| \leq 2(q-1)m+2s+2+ t_1 - t_3.$$ 
That leaves us with 
$$\sum_{i=0}^{t} |Y_i| \geq [(2q-1)m+s+1] - [2(q-1)m +2s +2+t_1-t_3]=(m-s-1)-(t_1-t_3).$$
Hence, 
$$\alpha_2(\phi) \geq f(v_0)+f(v_n) + \sum_{i=0}^{t} |Y_i| +t -1 \geq (m-s-2) + [f(v_0)+f(v_n)+t-t_1+t_3].$$
Thus, it is enough to show that 
\begin{equation}\label{eq enough}
[f(v_0)+f(v_n)+t-t_1+t_3] \geq 3.
\end{equation}
As $f(v_0)+f(v_n) \geq 1$, the Equation~(\ref{eq enough}) will be satisfied if there is one rightist sequence, or two balanced sequences. Furthermore, if $f(v_0)+f(v_n) \geq 2$, then 
 equation~(\ref{eq enough}) will be satisfied if there is one rightist or balanced sequence.

Notice that, if $r_{i1}$, for any $i \in \{1,2, \cdots, q\}$, is an element of an optimally colored pair, then the other element must be either $c_0$ or $l_{jm}$ for some $j \in \{1,2,\cdots,q\}$. 
We know that all right vertices, in particular, $r_{11}, r_{21}, \cdots, r_{q1}$, are part of some 
optimally colored sequences. 
Observe that, if they are distributed over two or more optimally colored sequences, then 
due to the above property, either one of those sequences will be rightist, or two of the sequences will be balanced. 

Moreover, if they are part of one optimally colored sequence $X_i$, 
then that sequence cannot be leftist. Furthermore, if the first or the last vertex of $X_i$ is $c_0$, then $X_i$ is rightist.  
Thus, in any case, equation~(\ref{eq enough}) is satisfied. Hence we are done. 
\end{proof}

Combining Lemmas \ref{lem tighter counting}, \ref{lem lower bound alpha1}, \ref{lem lower bound alpha2 even diam} and \ref{lem lower bound alpha2 odd diam}, therefore, we have the following lower bound for the parameter $rc_k(P_n^m)$.

\begin{lemma}\label{lem lower bound}
For all $k \geq diam(P_n^m)$, we have

$$rc_k(P_n^m)\geq
\begin{cases}
nk - \frac{n^2-m^2}{2m} &\text{ if } \lceil \frac{n}{m} \rceil \text{ is odd and } m|n,\\
nk - \frac{n^2-s^2}{2m} + 1 &\text{ if } \lceil \frac{n}{m} \rceil \text{ is odd and } m \nmid n,\\
nk - \frac{n^2}{2m} + 1 &\text{ if } \lceil \frac{n}{m}\rceil \text{ is even and } m|n,\\
nk - \frac{n^2-(m-s)^2}{2m} + 1 &\text{ if } \lceil \frac{n}{m} \rceil \text{ is even and } m \nmid n,
\end{cases}
$$
where $s \equiv n \pmod m$.  
\end{lemma}

\subsection{The upper bound}
Now let us prove the upper bound. 
We shall provide a radio $k$-coloring $\psi$ of $P_n^{m}$ and show that its span is the same as the
value of $rc_k(P_n^{m})$ stated in Theorem~\ref{th main}.
To define $\psi$, 
we shall use both the naming conventions. That is, we shall express the ordering 
$(v_0, v_1,  \cdots, v_n)$ of the vertices of $P_n^m$ with respect to $\psi$ in terms of the first naming convention.

Let us define a few ordering for the right (and similarly for the left) vertices:
\begin{enumerate}[(1)]
\item  $r_{ij} \prec_1 r_{i'j'}$ if either 
(i) $j < j'$ or (ii) $j= j'$ and $(-1)^{j-1} i < (-1)^{j'-1} i'$;

\item $r_{ij} \prec_2 r_{i'j'}$ if either 
(i) $j < j'$ or (ii) $j= j'$ and $(-1)^{m-j} i < (-1)^{m-j'} i'$;

\item $r_{ij} \prec_3 r_{i'j'}$ if either (i) $j < j'$ or (ii) $j= j'$ and $ i > i'$; and

\item $r_{ij} \prec_4 r_{i'j'}$ if either (i) $j < j'$ or (ii) $j= j'$ and $(-1)^{j} i < (-1)^{j'}i'$.
\end{enumerate}

Observe that, the orderings are based on comparing the second co-ordinate of the indices of the right (resp., left) vertices, and if they happen to be equal, then comparing the first co-ordinate of the indices with conditions on their parities. 
Moreover, all the above four orderings define total orders on the set of all right (resp., left) vertices. 
Thus, there is a unique increasing (resp., decreasing) sequence of 
right (or left) vertices with respect to $\prec_1$, $\prec_2$, $\prec_3$, and $\prec_4$.  Based on these orderings, we are going to 
construct a sequence of vertices of the graph and then greedy color the
vertices to provide our labeling. 

The sequences of the vertices are given as follows:
\begin{enumerate}[(1)]
\item An \textit{alternating chain} as a sequence of vertices of the form 
$(a_1, b_1, a_2, b_2, \cdots,$\\$a_{p}, b_{p})$ such that $(a_1, a_2, \cdots, a_p)$ is the increasing sequence of right vertices with respect to $\prec_1$ and 
$(b_1, b_2, \cdots, b_p)$ is the decreasing sequence of left vertices with respect to $\prec_2$.

\item A \textit{canonical chain} as a sequence of vertices of the form 
$(a_1, b_1, a_2, b_2, \cdots, a_{p}, \\ b_{p})$ such that $(a_1, a_2, \cdots, a_p)$ is the increasing sequence of right vertices with respect to $\prec_3$ and 
$(b_1, b_2, \cdots, b_p)$ is the decreasing sequence of left vertices with respect to $\prec_3$;

\item A \textit{special alternating chain} as a sequence of vertices of the form 
$(a_1, b_1, a_2, b_2, \\\cdots, a_{p}, b_{p})$ such that $(a_1, a_2, \cdots, a_p)$ is the increasing sequence of right vertices with respect to $\prec_2$ and 
$(b_1, b_2, \cdots, b_p)$ is the decreasing sequence of left vertices with respect to $\prec_1$; and

\item A \textit{special canonical chain} as a sequence of vertices of the form 
$(a_1, b_1, a_2, b_2, \\ \cdots, a_{p}, b_{p})$ such that $(a_1, a_2, \cdots, a_p)$ is the increasing sequence of right vertices with respect to $\prec_4$ and 
$(b_1, b_2, \cdots, b_p)$ is the decreasing sequence of left vertices with respect to $\prec_4$.

\end{enumerate}

\medskip

Notice that all the above four chains can exist only when the number of right and left vertices are equal. Of course, when $m|n$, all the chains exist. Otherwise, we shall modify the  names of the vertices a little to make them exist.

We are now ready to express the sequence $(v_0, v_1, \cdots, v_n)$ by splitting it into different cases which are depicted in Figures~\ref{fig:Case 1},~\ref{fig:Case 2},~\ref{fig:Case 3} and~\ref{fig:Case 4} for example. In the figures, both naming conventions for each of the vertices are given so that the reader may cross verify the correctness for that particular instance for each case. For convenience, also recall that 
$q = \lfloor \frac{diam(P_n^{m})}{2}  \rfloor$.

\medskip

\noindent \textit{Case 1: when $diam(P_n^{m})$ is even, $m|n$ and $k>diam(P_n^{m})$.} 
First of all, $(v_0, v_1, \cdots,\\ v_{2qm-1})$ is the alternating chain. Moreover, $v_n = c_0$.

\begin{figure}[h]
	\centering
	\begin{tikzpicture}[scale=.9]
	\foreach \x in {-8,...,8} \foreach \y in {0}{
				\node[circle, draw=black, scale=.3, fill=black] (\x) at (\x,\y){};}

	\node  at (-8,-.3) {\scriptsize $l_{24}$};
	\node  at (-7,-.3) {\scriptsize $l_{23}$};
	\node  at (-6,-.3) {\scriptsize $l_{22}$};
	\node  at (-5,-.3) {\scriptsize $l_{21}$};
	\node  at (-4,-.3) {\scriptsize $l_{14}$};
	\node  at (-3,-.3) {\scriptsize $l_{13}$};
	\node  at (-2,-.3) {\scriptsize $l_{12}$};
	\node  at (-1,-.3) {\scriptsize $l_{11}$};
	\node  at (0,-.3) {\scriptsize $c_{0}$};
	\node  at (1,-.3) {\scriptsize $r_{11}$};
	\node  at (2,-.3) {\scriptsize $r_{12}$};
	\node  at (3,-.3) {\scriptsize $r_{13}$};
	\node  at (4,-.3) {\scriptsize $r_{14}$};
	\node  at (5,-.3) {\scriptsize $r_{21}$};
	\node  at (6,-.3) {\scriptsize $r_{22}$};
	\node  at (7,-.3) {\scriptsize $r_{23}$};
	\node  at (8,-.3) {\scriptsize $r_{24}$};

    \node  at (-8,.3) {\scriptsize $4$};
	\node  at (-7,.3) {\scriptsize $28$};
	\node  at (-6,.3) {\scriptsize $36$};
	\node  at (-5,.3) {\scriptsize $60$};
	\node  at (-4,.3) {\scriptsize $11$};
	\node  at (-3,.3) {\scriptsize $19$};
	\node  at (-2,.3) {\scriptsize $43$};
	\node  at (-1,.3) {\scriptsize $51$};
	\node  at (0,.3) {\scriptsize $65$};
	\node  at (1,.3) {\scriptsize $0$};
	\node  at (2,.3) {\scriptsize $24$};
	\node  at (3,.3) {\scriptsize $32$};
	\node  at (4,.3) {\scriptsize $56$};
	\node  at (5,.3) {\scriptsize $7$};
	\node  at (6,.3) {\scriptsize $15$};
	\node  at (7,.3) {\scriptsize $39$};
	\node  at (8,.3) {\scriptsize $47$};

        \node  at (-8,-.8) {\scriptsize $v_{1}$};
	\node  at (-7,-.8) {\scriptsize $v_{7}$};
	\node  at (-6,-.8) {\scriptsize $v_{9}$};
	\node  at (-5,-.8) {\scriptsize $v_{15}$};
	\node  at (-4,-.8) {\scriptsize $v_{3}$};
	\node  at (-3,-.8) {\scriptsize $v_{5}$};
	\node  at (-2,-.8) {\scriptsize $v_{11}$};
	\node  at (-1,-.8) {\scriptsize $v_{13}$};
	\node  at (0,-.8) {\scriptsize $v_{16}$};
	\node  at (1,-.8) {\scriptsize $v_{0}$};
	\node  at (2,-.8) {\scriptsize $v_{6}$};
	\node  at (3,-.8) {\scriptsize $v_{8}$};
	\node  at (4,-.8) {\scriptsize $v_{14}$};
	\node  at (5,-.8) {\scriptsize $v_{2}$};
	\node  at (6,-.8) {\scriptsize $v_{4}$};
	\node  at (7,-.8) {\scriptsize $v_{10}$};
	\node  at (8,-.8) {\scriptsize $v_{12}$};

	\foreach \x/\y in {-8/-7,-7/-6,-6/-5,-5/-4,-4/-3,-3/-2,-2/-1,-1/0,0/1,1/2,2/3,3/4,4/5,5/6,6/7,7/8} 
	\draw[-,draw=black, opacity=.5] (\x) -- (\y);

	\foreach \x/\y/\z/\w in {-8/-4/-7/-5,-4/0/-3/-1,0/4/1/3,4/8/5/7,
	-7/-3/-6/-4,-3/1/-2/0,1/5/2/4,
    -6/-2/-5/-3,-2/2/-1/1,2/6/3/5,
    -5/-1/-4/-2,-1/3/0/2,3/7/4/6}  
	\draw[-,draw=gray, opacity=.2,line width=0.05mm] (\x) .. controls (\z,1.5) and (\w,1.5) .. (\y);

	\foreach \x/\y/\z/\w in {-6/-3/-5/-4,-3/0/-2/-1,0/3/1/2,3/6/4/5,
		-8/-5/-7/-6,-5/-2/-4/-3,-2/1/-1/0,1/4/2/3,4/7/5/6,
		-7/-4/-6/-5,-4/-1/-3/-2,-1/2/0/1,2/5/3/4,5/8/6/7}  
	\draw[-,draw=gray, opacity=.2,line width=0.05mm] (\x) .. controls (\z,.9) and (\w,.9) .. (\y);

	\foreach \x/\y/\z in {-8/-6/-7,-6/-4/-5,-4/-2/-3,-2/0/-1,0/2/1,2/4/3,4/6/5,6/8/7,
		-7/-5/-6,-5/-3/-4,-3/-1/-2,-1/1/0,1/3/2,3/5/4,5/7/6}
	\draw[-,draw=gray, opacity=.2,line width=0.05mm] (\x) .. controls (\z,.4) .. (\y);

	\draw[blue, thin] (-.35,-1) rectangle (.35,.5);
   	\draw[red,thin] (.65,-1) rectangle (4.35,.5);
   	\draw[red,thin] (-.65,-1) rectangle (-4.35,.5);
   	\draw[brown(traditional),thin] (4.65,-1) rectangle (8.35,.5);
	\draw[brown(traditional),thin] (-4.65,-1) rectangle (-8.35,.5);
	
    \node  at (0,-1.3) {\scriptsize \textcolor{blue}{$L_0$}};
	\node  at (2.5,-1.3) {\scriptsize \textcolor{red}{$L_1$}};
	\node  at (-2.5,-1.3) {\scriptsize \textcolor{red}{$L_1$}};
	\node  at (6.5,-1.3) {\scriptsize \textcolor{brown(traditional)}{$L_2$}};
	\node  at (-6.5,-1.3) {\scriptsize \textcolor{brown(traditional)}{$L_2$}};

\end{tikzpicture}
\caption{\textit{Case 1.} $n=16$, $m=4$, $diam(P_{16}^4)=4$, $k=6$.}
\label{fig:Case 1}
\end{figure}

\medskip

\noindent \textit{Case 2: when $diam(P_n^{m})$ is odd, $m|n$ and $k> diam(P_n^{m})$.}
Let the ordering of the vertices be $(v_0, v_1, \cdots, v_{2qm+m})$. Now, $v_{j(2q+1)} = c_j$ for all $0 \leq j \leq m$. The remaining vertices follow the canonical chain.

\begin{figure}[h]
	\centering
	\begin{tikzpicture}[scale=.77]
		\foreach \x in {-10,...,10} \foreach \y in {0}{
			\node[circle, draw=black, scale=.3, fill=black] (\x) at (\x,\y){};}

		\node  at (-10,-.3) {\scriptsize $l_{24}$};
		\node  at (-9,-.3) {\scriptsize $l_{23}$};
		\node  at (-8,-.3) {\scriptsize $l_{22}$};
		\node  at (-7,-.3) {\scriptsize $l_{21}$};
		\node  at (-6,-.3) {\scriptsize $l_{14}$};
		\node  at (-5,-.3) {\scriptsize $l_{13}$};
		\node  at (-4,-.3) {\scriptsize $l_{12}$};
		\node  at (-3,-.3) {\scriptsize $l_{11}$};
		\node  at (-2,-.3) {\scriptsize $c_{0}$};
		\node  at (-1,-.3) {\scriptsize $c_{1}$};
		\node  at (0,-.3) {\scriptsize $c_{2}$};
		\node  at (1,-.3) {\scriptsize $c_{3}$};
		\node  at (2,-.3) {\scriptsize $c_{4}$};
		\node  at (3,-.3) {\scriptsize $r_{11}$};
		\node  at (4,-.3) {\scriptsize $r_{12}$};
		\node  at (5,-.3) {\scriptsize $r_{13}$};
		\node  at (6,-.3) {\scriptsize $r_{14}$};
		\node  at (7,-.3) {\scriptsize $r_{21}$};
		\node  at (8,-.3) {\scriptsize $r_{22}$};
		\node  at (9,-.3) {\scriptsize $r_{23}$};
		\node  at (10,-.3) {\scriptsize $r_{24}$};

        \node  at (-10,.3) {\scriptsize $18$};
		\node  at (-9,.3) {\scriptsize $41$};
		\node  at (-8,.3) {\scriptsize $64$};
		\node  at (-7,.3) {\scriptsize $87$};
		\node  at (-6,.3) {\scriptsize $9$};
		\node  at (-5,.3) {\scriptsize $32$};
		\node  at (-4,.3) {\scriptsize $55$};
		\node  at (-3,.3) {\scriptsize $78$};
		\node  at (-2,.3) {\scriptsize $0$};
		\node  at (-1,.3) {\scriptsize $23$};
		\node  at (0,.3) {\scriptsize $46$};
		\node  at (1,.3) {\scriptsize $69$};
		\node  at (2,.3) {\scriptsize $92$};
		\node  at (3,.3) {\scriptsize $14$};
		\node  at (4,.3) {\scriptsize $37$};
		\node  at (5,.3) {\scriptsize $60$};
		\node  at (6,.3) {\scriptsize $83$};
		\node  at (7,.3) {\scriptsize $5$};
		\node  at (8,.3) {\scriptsize $28$};
		\node  at (9,.3) {\scriptsize $51$};
		\node  at (10,.3) {\scriptsize $74$};

		\node  at (-10,-.8) {\scriptsize $v_{4}$};
		\node  at (-9,-.8) {\scriptsize $v_{9}$};
		\node  at (-8,-.8) {\scriptsize $v_{14}$};
		\node  at (-7,-.8) {\scriptsize $v_{19}$};
		\node  at (-6,-.8) {\scriptsize $v_{2}$};
		\node  at (-5,-.8) {\scriptsize $v_{7}$};
		\node  at (-4,-.8) {\scriptsize $v_{12}$};
		\node  at (-3,-.8) {\scriptsize $v_{17}$};
		\node  at (-2,-.8) {\scriptsize $v_{0}$};
		\node  at (-1,-.8) {\scriptsize $v_{5}$};
		\node  at (0,-.8) {\scriptsize $v_{10}$};
		\node  at (1,-.8) {\scriptsize $v_{15}$};
		\node  at (2,-.8) {\scriptsize $v_{20}$};
		\node  at (3,-.8) {\scriptsize $v_{3}$};
		\node  at (4,-.8) {\scriptsize $v_{8}$};
		\node  at (5,-.8) {\scriptsize $v_{13}$};
		\node  at (6,-.8) {\scriptsize $v_{18}$};
		\node  at (7,-.8) {\scriptsize $v_{1}$};
		\node  at (8,-.8) {\scriptsize $v_{6}$};
		\node  at (9,-.8) {\scriptsize $v_{11}$};
		\node  at (10,-.8) {\scriptsize $v_{16}$};

		\foreach \x/\y in {-10/-9,-9/-8,-8/-7,-7/-6,-6/-5,-5/-4,-4/-3,-3/-2,-2/-1,-1/0,0/1,1/2,2/3,3/4,4/5,5/6,6/7,7/8,8/9,9/10} 
		\draw[-,draw=black, opacity=.5] (\x) -- (\y);

		\foreach \x/\y/\z/\w in {-8/-4/-7/-5,-4/0/-3/-1,0/4/1/3,4/8/5/7,
			-7/-3/-6/-4,-3/1/-2/0,1/5/2/4,5/9/6/8,
			-10/-6/-9/-7,-6/-2/-5/-3,-2/2/-1/1,2/6/3/5,6/10/7/9,
			-9/-5/-8/-6,-5/-1/-4/-2,-1/3/0/2,3/7/4/6}  
		\draw[-,draw=gray, opacity=.2,line width=0.05mm] (\x) .. controls (\z,1.5) and (\w,1.5) .. (\y);

		\foreach \x/\y/\z/\w in {-9/-6/-8/-7,-6/-3/-5/-4,-3/0/-2/-1,0/3/1/2,3/6/4/5,6/9/7/8,
			-8/-5/-7/-6,-5/-2/-4/-3,-2/1/-1/0,1/4/2/3,4/7/5/6,7/10/8/9,
			-10/-7/-9/-8,-7/-4/-6/-5,-4/-1/-3/-2,-1/2/0/1,2/5/3/4,5/8/6/7}  
		\draw[-,draw=gray, opacity=.2,line width=0.05mm] (\x) .. controls (\z,.9) and (\w,.9) .. (\y);

		\foreach \x/\y/\z in {-10/-8/-9,-8/-6/-7,-6/-4/-5,-4/-2/-3,-2/0/-1,0/2/1,2/4/3,4/6/5,6/8/7,8/10/9,
			-9/-7/-8,-7/-5/-6,-5/-3/-4,-3/-1/-2,-1/1/0,1/3/2,3/5/4,5/7/6,7/9/8}
		\draw[-,draw=gray, opacity=.2,line width=0.05mm] (\x) .. controls (\z,.4) .. (\y);

		\draw[blue, thin] (-2.35,-1) rectangle (2.35,.5);
		\draw[red,thin] (2.65,-1) rectangle (6.35,.5);
		\draw[red,thin] (-2.65,-1) rectangle (-6.35,.5);
		\draw[brown(traditional),thin] (6.65,-1) rectangle (10.35,.5);
		\draw[brown(traditional),thin] (-6.65,-1) rectangle (-10.35,.5);

		\node  at (0,-1.3) {\scriptsize \textcolor{blue}{$L_0$}};
		\node  at (4.5,-1.3) {\scriptsize \textcolor{red}{$L_1$}};
		\node  at (-4.5,-1.3) {\scriptsize \textcolor{red}{$L_1$}};
		\node  at (8.5,-1.3) {\scriptsize \textcolor{brown(traditional)}{$L_2$}};
		\node  at (-8.5,-1.3) {\scriptsize \textcolor{brown(traditional)}{$L_2$}};

	\end{tikzpicture}
	\caption{\textit{Case 2.} $n=20$, $m=4$, $diam(P_{20}^4)=5$, $k=7$.}
	\label{fig:Case 2}
\end{figure}

\medskip

\noindent \textit{Case 3: when $diam(P_n^{m})$ is odd, $m\nmid n$ and $k>diam(P_n^{m})$.}
For any set $A$, let $A^\star$ represent an ordered sequence of the elements of $A$.
Let $G \cong P_n^m$ and 
$S = V(G) = \{ v_0, v_1, v_2, \cdots , v_{2qm+s}\}$. Then $S^\star$ is defined as described.
First, define
$$T = \{ v_t : 0 \leq t \leq s(2q+1) \} - \{ v_{j(2q+1)} : 0 \leq j \leq s \}.$$
Order $T^\star$ as a canonical chain. Also, define $v_{j(2q+1)} = c_j$ for all $0 \leq j \leq s$. 
Assume $G'$ to be the subgraph of $G$ induced by the subset $S - \{ r_{q1}, r_{q2}, \cdots , r_{qs} \}$ of $S$. Then $G' \cong P_{n'}^m$, $m \vert n'$ and $diam(G') = \frac{n'}{m}$ is even, where $n'=n-s$. 
Define
$$v_n = l_{11} \text{ and } U = \{ v_t : s(2q+1) + 1 \leq t < n \}.$$
Note that $U \subset V(G')$. 
Order $U^\star$ (as vertices of $G'$) by the following.
\begin{enumerate}[(i)]
    \item Special alternating chain when $m$ and $s$ have the same parity.
    
    \item Alternating chain when $m$ is even and $s$ is odd.
    
    \item Special canonical chain when $m$ is odd and $s$ is even.
\end{enumerate}

\begin{figure}[h]
	\centering
	\begin{tikzpicture}[scale=.8]
		\foreach \x in {-10,...,9} \foreach \y in {0}{
			\node[circle, draw=black, scale=.3, fill=black] (\x) at (\x,\y){};}

		\node  at (-10,-.3) {\scriptsize $l_{24}$};
		\node  at (-9,-.3) {\scriptsize $l_{23}$};
		\node  at (-8,-.3) {\scriptsize $l_{22}$};
		\node  at (-7,-.3) {\scriptsize $l_{21}$};
		\node  at (-6,-.3) {\scriptsize $l_{14}$};
		\node  at (-5,-.3) {\scriptsize $l_{13}$};
		\node  at (-4,-.3) {\scriptsize $l_{12}$};
		\node  at (-3,-.3) {\scriptsize $l_{11}$};
		\node  at (-2,-.3) {\scriptsize $c_{0}$};
		\node  at (-1,-.3) {\scriptsize $c_{1}$};
		\node  at (0,-.3) {\scriptsize $c_{2}$};
		\node  at (1,-.3) {\scriptsize $c_{3}$};
		\node  at (2,-.3) {\scriptsize $c_{4}$};
		\node  at (3,-.3) {\scriptsize $r_{11}$};
		\node  at (4,-.3) {\scriptsize $r_{12}$};
		\node  at (5,-.3) {\scriptsize $r_{13}$};
		\node  at (6,-.3) {\scriptsize $r_{14}$};
		\node  at (7,-.3) {\scriptsize $r_{21}$};
		\node  at (8,-.3) {\scriptsize $r_{22}$};
		\node  at (9,-.3) {\scriptsize $r_{23}$};

        \node  at (-10,.3) {\scriptsize $18$};
		\node  at (-9,.3) {\scriptsize $41$};
		\node  at (-8,.3) {\scriptsize $64$};
		\node  at (-7,.3) {\scriptsize $79$};
		\node  at (-6,.3) {\scriptsize $9$};
		\node  at (-5,.3) {\scriptsize $32$};
		\node  at (-4,.3) {\scriptsize $55$};
		\node  at (-3,.3) {\scriptsize $90$};
		\node  at (-2,.3) {\scriptsize $0$};
		\node  at (-1,.3) {\scriptsize $23$};
		\node  at (0,.3) {\scriptsize $46$};
		\node  at (1,.3) {\scriptsize $69$};
		\node  at (2,.3) {\scriptsize $84$};
		\node  at (3,.3) {\scriptsize $14$};
		\node  at (4,.3) {\scriptsize $37$};
		\node  at (5,.3) {\scriptsize $60$};
		\node  at (6,.3) {\scriptsize $75$};
		\node  at (7,.3) {\scriptsize $5$};
		\node  at (8,.3) {\scriptsize $28$};
		\node  at (9,.3) {\scriptsize $51$};

		\node  at (-10,-.8) {\scriptsize $v_{4}$};
		\node  at (-9,-.8) {\scriptsize $v_{9}$};
		\node  at (-8,-.8) {\scriptsize $v_{14}$};
		\node  at (-7,-.8) {\scriptsize $v_{17}$};
		\node  at (-6,-.8) {\scriptsize $v_{2}$};
		\node  at (-5,-.8) {\scriptsize $v_{7}$};
		\node  at (-4,-.8) {\scriptsize $v_{12}$};
		\node  at (-3,-.8) {\scriptsize $v_{19}$};
		\node  at (-2,-.8) {\scriptsize $v_{0}$};
		\node  at (-1,-.8) {\scriptsize $v_{5}$};
		\node  at (0,-.8) {\scriptsize $v_{10}$};
		\node  at (1,-.8) {\scriptsize $v_{15}$};
		\node  at (2,-.8) {\scriptsize $v_{18}$};
		\node  at (3,-.8) {\scriptsize $v_{3}$};
		\node  at (4,-.8) {\scriptsize $v_{8}$};
		\node  at (5,-.8) {\scriptsize $v_{13}$};
		\node  at (6,-.8) {\scriptsize $v_{16}$};
		\node  at (7,-.8) {\scriptsize $v_{1}$};
		\node  at (8,-.8) {\scriptsize $v_{6}$};
		\node  at (9,-.8) {\scriptsize $v_{11}$};

		\foreach \x/\y in {-10/-9,-9/-8,-8/-7,-7/-6,-6/-5,-5/-4,-4/-3,-3/-2,-2/-1,-1/0,0/1,1/2,2/3,3/4,4/5,5/6,6/7,7/8,8/9} 
		\draw[-,draw=black, opacity=.5] (\x) -- (\y);

		\foreach \x/\y/\z/\w in {-8/-4/-7/-5,-4/0/-3/-1,0/4/1/3,4/8/5/7,
			-7/-3/-6/-4,-3/1/-2/0,1/5/2/4,5/9/6/8,
			-10/-6/-9/-7,-6/-2/-5/-3,-2/2/-1/1,2/6/3/5,
			-9/-5/-8/-6,-5/-1/-4/-2,-1/3/0/2,3/7/4/6}  
		\draw[-,draw=gray, opacity=.2,line width=0.05mm] (\x) .. controls (\z,1.5) and (\w,1.5) .. (\y);

		\foreach \x/\y/\z/\w in {-9/-6/-8/-7,-6/-3/-5/-4,-3/0/-2/-1,0/3/1/2,3/6/4/5,6/9/7/8,
		-8/-5/-7/-6,-5/-2/-4/-3,-2/1/-1/0,1/4/2/3,4/7/5/6,
		-10/-7/-9/-8,-7/-4/-6/-5,-4/-1/-3/-2,-1/2/0/1,2/5/3/4,5/8/6/7}  
		\draw[-,draw=gray, opacity=.2,line width=0.05mm] (\x) .. controls (\z,.9) and (\w,.9) .. (\y);

		\foreach \x/\y/\z in {-10/-8/-9,-8/-6/-7,-6/-4/-5,-4/-2/-3,-2/0/-1,0/2/1,2/4/3,4/6/5,6/8/7,
		-9/-7/-8,-7/-5/-6,-5/-3/-4,-3/-1/-2,-1/1/0,1/3/2,3/5/4,5/7/6,7/9/8}
		\draw[-,draw=gray, opacity=.2,line width=0.05mm] (\x) .. controls (\z,.4) .. (\y);

		\draw[blue, thin] (-2.35,-1) rectangle (2.35,.5);
		\draw[red,thin] (2.65,-1) rectangle (6.35,.5);
		\draw[red,thin] (-2.65,-1) rectangle (-6.35,.5);
		\draw[brown(traditional),thin] (6.65,-1) rectangle (9.35,.5);
		\draw[brown(traditional),thin] (-6.65,-1) rectangle (-10.35,.5);

		\node  at (0,-1.3) {\scriptsize \textcolor{blue}{$L_0$}};
		\node  at (4.5,-1.3) {\scriptsize \textcolor{red}{$L_1$}};
		\node  at (-4.5,-1.3) {\scriptsize \textcolor{red}{$L_1$}};
		\node  at (8,-1.3) {\scriptsize \textcolor{brown(traditional)}{$L_2$}};
		\node  at (-8.5,-1.3) {\scriptsize \textcolor{brown(traditional)}{$L_2$}};

	\end{tikzpicture}
	\caption{\textit{Case 3.} $n=19$, $m=4$, $diam(P_{19}^4)=5$, $k=7$, $s=3$.}
	\label{fig:Case 3}
\end{figure}

\medskip

\noindent \textit{Case 4: when $diam(P_n^{m})$ is even, $m \nmid n$ and $k>diam(P_n^{m})$.} 
Notice that, in this case, the left vertices are $(m-s)$ more than the right vertices. Also, $L_0$ has 
only one vertex in this case. We shall discard some vertices from the set of left vertices, and then present the ordering. 
To be specific, we disregard the subset $\{l_{11}, l_{12}, \cdots, l_{1(m-s)}\}$, temporarily, 
from the set of left vertices and consider the 
alternating chain. 
First of all, $(v_0, v_1, \cdots, v_{2qm-2m+2s-1})$ is the  alternating chain. Additionally, we have
$$\left(v_{2qm-2m+2s}, v_{2qm-2m+2s+1}, v_{2qm-2m+2s+2}, \cdots, v_{2qm-m+s}\right) = (c_0, l_{11}, l_{12}, \cdots, l_{1(m-s)}).$$

\begin{figure}[h]
        \centering
	\begin{tikzpicture}[scale=.8]
	
		\foreach \x in {-8,...,6} \foreach \y in {0}{
			\node[circle, draw=black, scale=.3, fill=black] (\x) at (\x,\y){};}

		\node  at (-8,-.3) {\scriptsize $l_{24}$};
		\node  at (-7,-.3) {\scriptsize $l_{23}$};
		\node  at (-6,-.3) {\scriptsize $l_{22}$};
		\node  at (-5,-.3) {\scriptsize $l_{21}$};
		\node  at (-4,-.3) {\scriptsize $l_{14}$};
		\node  at (-3,-.3) {\scriptsize $l_{13}$};
		\node  at (-2,-.3) {\scriptsize $l_{12}$};
		\node  at (-1,-.3) {\scriptsize $l_{11}$};
		\node  at (0,-.3) {\scriptsize $c_{0}$};
		\node  at (1,-.3) {\scriptsize $r_{11}$};
		\node  at (2,-.3) {\scriptsize $r_{12}$};
		\node  at (3,-.3) {\scriptsize $r_{13}$};
		\node  at (4,-.3) {\scriptsize $r_{14}$};
		\node  at (5,-.3) {\scriptsize $r_{21}$};
		\node  at (6,-.3) {\scriptsize $r_{22}$};

    	\node  at (-8,.3) {\scriptsize $4$};
		\node  at (-7,.3) {\scriptsize $28$};
		\node  at (-6,.3) {\scriptsize $36$};
		\node  at (-5,.3) {\scriptsize $44$};
		\node  at (-4,.3) {\scriptsize $11$};
		\node  at (-3,.3) {\scriptsize $19$};
		\node  at (-2,.3) {\scriptsize $61$};
		\node  at (-1,.3) {\scriptsize $55$};
		\node  at (0,.3) {\scriptsize $49$};
		\node  at (1,.3) {\scriptsize $0$};
		\node  at (2,.3) {\scriptsize $24$};
		\node  at (3,.3) {\scriptsize $32$};
		\node  at (4,.3) {\scriptsize $40$};
		\node  at (5,.3) {\scriptsize $7$};
		\node  at (6,.3) {\scriptsize $15$};

		\node  at (-8,-.8) {\scriptsize $v_{1}$};
		\node  at (-7,-.8) {\scriptsize $v_{7}$};
		\node  at (-6,-.8) {\scriptsize $v_{9}$};
		\node  at (-5,-.8) {\scriptsize $v_{11}$};
		\node  at (-4,-.8) {\scriptsize $v_{3}$};
		\node  at (-3,-.8) {\scriptsize $v_{5}$};
		\node  at (-2,-.8) {\scriptsize $v_{14}$};
		\node  at (-1,-.8) {\scriptsize $v_{13}$};
		\node  at (0,-.8) {\scriptsize $v_{12}$};
		\node  at (1,-.8) {\scriptsize $v_{0}$};
		\node  at (2,-.8) {\scriptsize $v_{6}$};
		\node  at (3,-.8) {\scriptsize $v_{8}$};
		\node  at (4,-.8) {\scriptsize $v_{10}$};
		\node  at (5,-.8) {\scriptsize $v_{2}$};
		\node  at (6,-.8) {\scriptsize $v_{4}$};

		\foreach \x/\y in {-8/-7,-7/-6,-6/-5,-5/-4,-4/-3,-3/-2,-2/-1,-1/0,0/1,1/2,2/3,3/4,4/5,5/6} 
		\draw[-,draw=black, opacity=.5] (\x) -- (\y);

		\foreach \x/\y/\z/\w in {-8/-4/-7/-5,-4/0/-3/-1,0/4/1/3,
			-7/-3/-6/-4,-3/1/-2/0,1/5/2/4,
			-6/-2/-5/-3,-2/2/-1/1,2/6/3/5,-5/-1/-4/-2,-1/3/0/2}  
		\draw[-,draw=gray, opacity=.2,line width=0.05mm] (\x) .. controls (\z,1.5) and (\w,1.5) .. (\y);

		\foreach \x/\y/\z/\w in {-6/-3/-5/-4,-3/0/-2/-1,0/3/1/2,3/6/4/5,		-8/-5/-7/-6,-5/-2/-4/-3,-2/1/-1/0,1/4/2/3,
		-7/-4/-6/-5,-4/-1/-3/-2,-1/2/0/1,2/5/3/4}  
		\draw[-,draw=gray, opacity=.2,line width=0.05mm] (\x) .. controls (\z,.9) and (\w,.9) .. (\y);

		\foreach \x/\y/\z in {-8/-6/-7,-6/-4/-5,-4/-2/-3,-2/0/-1,0/2/1,2/4/3,4/6/5,
		-7/-5/-6,-5/-3/-4,-3/-1/-2,-1/1/0,1/3/2,3/5/4}
		\draw[-,draw=gray, opacity=.2,line width=0.05mm] (\x) .. controls (\z,.4) .. (\y);

		\draw[blue, thin] (-.35,-1) rectangle (.35,.5);
		\draw[red,thin] (.65,-1) rectangle (4.35,.5);
		\draw[red,thin] (-.65,-1) rectangle (-4.35,.5);
		\draw[brown(traditional),thin] (4.65,-1) rectangle (6.35,.5);
		\draw[brown(traditional),thin] (-4.65,-1) rectangle (-8.35,.5);

		\node  at (0,-1.3) {\scriptsize \textcolor{blue}{$L_0$}};
		\node  at (2.5,-1.3) {\scriptsize \textcolor{red}{$L_1$}};
		\node  at (-2.5,-1.3) {\scriptsize \textcolor{red}{$L_1$}};
		\node  at (5.5,-1.3) {\scriptsize \textcolor{brown(traditional)}{$L_2$}};
		\node  at (-6.5,-1.3) {\scriptsize \textcolor{brown(traditional)}{$L_2$}};

	\end{tikzpicture}
    \caption{\textit{Case 4.} $n=14$, $m=4$, $diam(P_{14}^4)=4$, $k=6$, $s=2$.}
    \label{fig:Case 4}
\end{figure}

\bigskip

Thus, we have obtained a sequence $(v_0, v_1, \cdots, v_n)$ in each case under consideration. Now, we define, $\psi(v_0)=0$ and $\psi(v_{i+1}) = \psi(v_{i}) + k+1 - d(v_i, v_{i+1})$, 
recursively, for all $i \in \{1, 2, \cdots, n-1\}$. 
Next, we  note that $\psi$ is a radio $k$-coloring. 

\medskip

\begin{lemma}\label{lem valid coloring}
The function $\psi$ is a radio $k$-coloring of $P_n^{m}$. 
\end{lemma}

\begin{proof}
Notice that, the way $\psi$ is defined, for all $i \in \{0,1, \cdots, n-1\}$, we have 
$\psi(v_{i+1})-\psi(v_{i}) =  k+1 - d(v_i, v_{i+1})$. 
Furthermore, one can observe that for all 
$i \in \{0, 1, \cdots, n-2\}$, we have 
$\psi(v_{i+2})-\psi(v_{i}) \geq  k$. As the value of the image of $\psi$ increases with respect to the indices of $v_i$s, $\psi$ satisfies the conditions for being a radio $k$-coloring. 
\end{proof}

This brings us to the upper bound for $rc_k(P_n^m)$. 

\begin{lemma}\label{lem upper bound}
For all $k > diam(P_n^{m})$, we have
$$rc_k(P_n^m) \le 
\begin{cases}
nk - \frac{n^2-m^2}{2m} &\text{ if } \lceil \frac{n}{m} \rceil \text{ is odd and } m|n,\\
nk - \frac{n^2-s^2}{2m} + 1 &\text{ if } \lceil \frac{n}{m} \rceil \text{ is odd and } m \nmid n,\\
nk - \frac{n^2}{2m} + 1 &\text{ if } \lceil \frac{n}{m}\rceil \text{ is even and } m|n,\\
nk - \frac{n^2-(m-s)^2}{2m} + 1 &\text{ if } \lceil \frac{n}{m} \rceil \text{ is even and } m \nmid n,
\end{cases}
$$
where $s \equiv n \pmod m$.  
\end{lemma}

\begin{proof}
Observe that, $rc_k(P_n^m)\le span(\psi)$. So, to prove the upper bound, it is enough to show that for all $k>diam(P_n^m)$ and $s \equiv n \pmod m$, 
$$span(\psi) = \begin{cases}
nk - \frac{n^2-m^2}{2m} &\text{ if } \lceil \frac{n}{m} \rceil \text{ is odd and } m|n,\\
nk - \frac{n^2-s^2}{2m} + 1 &\text{ if } \lceil \frac{n}{m} \rceil \text{ is odd and } m \nmid n,\\
nk - \frac{n^2}{2m} + 1 &\text{ if } \lceil \frac{n}{m}\rceil \text{ is even and } m|n,\\
nk - \frac{n^2-(m-s)^2}{2m} + 1 &\text{ if } \lceil \frac{n}{m} \rceil \text{ is even and } m \nmid n.
\end{cases}
$$ 
Notice that, for odd values of $diam(P_n^{m})$ and for even values of $diam(P_n^{m})$ where $m|n$, the whole sequence $(v_0, v_1, \cdots, v_n)$ is optimally colored with respect to $\psi$. Moreover, note that
$$f(v_0)+f(v_n) =
\begin{cases}
0 &\text{ if } \lceil \frac{n}{m} \rceil \text{ is odd and } m|n,\\
1 &\text{ if } \lceil \frac{n}{m} \rceil \text{ is odd and } m \nmid n,\\
1 &\text{ if } \lceil \frac{n}{m}\rceil \text{ is even and } m|n.
\end{cases}
$$
Thus, adding these values with $\alpha_1$ (from Lemma~\ref{lem lower bound alpha1}) will complete the proof for the first three cases.

For the final case, that is, for even values of $diam(P_n^{m})$ where $m \nmid n$, the sequence $(v_0, v_1,$ $\cdots, v_{2(q-1)m+2s+1})$ is an optimally colored sequence. On the other hand,\\$(v_{2(q-1)m+2s+2}, v_{2(q-1)m+2s+3},$ $\cdots, v_n)$ is a loosely colored sequence. Thus, the whole sequence has exactly $(m-s-1)$ loosely colored pairs, namely, \\
$(v_{2(q-1)m+2s+1}, v_{2(q-1)m+2s+2})$, $(v_{2(q-1)m+2s+2}, v_{2(q-1)m+2s+3})$, $\cdots$, $(v_{n-1}, v_n)$. These pairs are nothing but  
$(l_{11}, l_{12})$, $(l_{12}, l_{13})$, $\cdots$, $(l_{1(m-s-1)}, l_{1(m-s)})$. Now, let us count how many extra colors are skipped for each pair. In fact, we claim that the number of extra colors skipped for the pair 
$(l_{1i}, l_{1(i+1)})$ is one, for all $i \in \{1, 2, \cdots, m-s-1\}$.
Notice that, both $l_{1i}$ and $l_{1(i+1)}$ are from $L_1$. Thus, if they were optimally colored, we would have had 
$$\psi(l_{1(i+1)}) = \psi(l_{1i})+k+1 - f(l_{1(i+1)}) - f(l_{1i}) = \psi(l_{1i})+k-1.$$
However, the distance between $l_{1i}$ and $l_{1(i+1)}$ is one. Thus, what we actually have is
$$\psi(l_{1(i+1)}) = \psi(l_{1i})+k+1 - d(l_{1(i+1)}, l_{1i}) = \psi(l_{1i})+k.$$
Thus, a total of extra $(m-s-1)$ colors are skipped while coloring the said loosely colored sequence. 
Moreover, as $f(v_0)+f(v_n) = 2$ in this case, we have $span(\psi) = \alpha_1 + (m-s-1)+2$. 
Hence, simply replacing the value of $\alpha_1$ from Lemma~\ref{lem lower bound alpha1}  in the above equation ends the proof. 
\end{proof}

\subsection{The proofs}
Finally we are ready to conclude the proofs. 

\medskip

\noindent  \textit{Proof of Theorem~\ref{th main}} The proof follows directly from Lemmas \ref{lem lower bound} and \ref{lem upper bound}. \qed

\bigskip

\noindent  \textit{Proof of Theorem~\ref{th algo}}  Notice that the 
proof of the upper bound for Theorem~\ref{th main} is given by prescribing an algorithm (implicitly). The algorithm requires 
ordering the vertices of the input graph, and then providing the coloring based on the ordering. Each step runs in linear order of the number of vertices in the input graph. 
Moreover, we have theoretically proved the tightness of the upper bound. Thus, we are done. \qed

\section{Concluding remarks}\label{conclusion}
\begin{enumerate}[(1)]
    \item Note that when $m=1$, that is, for paths, our result coincides with that of Kchikech, Khennoufa and Togni~\cite{kchikech2007linear} when $k>diam(P_n)$.

    \item We have found the exact value of radio $k$-chromatic number of powers of paths $P_n^m$ having small diameters, that is when $k>diam(P_n^m)$. Also, for $k=diam(P_n^m)$, the exact value of $rc_k(P_n^m)$ has been found~\cite{rao2018radio}. Although, finding the exact value of $rc_k(P_n^m)$ when $k< diam(P_n^m)$ still remains open.

\begin{question}
Find the exact value of $rc_k(P_n^m)$ when $k<diam(P_n^m)$.
\end{question}

\item Saha and Panigrahi~\cite{saha2012antipodal} provided lower and upper bounds for $rc_k(C_n^m)$ when $k< diam(C_n^m)$. 
We believe that our lower bound technique can be modified and used to find radio $k$-chromatic number of cycle powers $C_n^m$ when $k>diam(C_n^m)$. Here, $m$-th power of cycle $C_n^m$ can be obtained by adding edges between the vertices of cycle $C_n$ that are at most $m$ distance apart.

\begin{question}
Determine the exact value of $rc_k(C_n^m)$ for all values of $k$.
\end{question}

\item It is well-known from Brook's theorem~\cite{brooks1941colouring} that for a connected graph $G$, $rc_1(G)\le \Delta-1$ (since $\chi(G)=rc_1(G)-1$), unless $G$ is a complete graph or an odd cycle.
As mentioned earlier, for a graph $G$ with maximum degree $\Delta$, Griggs and Yeh~\cite{griggs1992labelling} conjectured that $rc_2(G)\le \Delta^2$. An intuitive question is to ask whether this could be extended for any value of $k$. Note that, using Brook's Theorem, one can attain a trivial upper bound of $rc_k(G) \leq k\Delta^k$.

\begin{question}
    For a graph $G$ with maximum degree $\Delta$, what is the minimum $r$ such that $rc_k(G)=O(\Delta^r)$?
\end{question}

\item It is not difficult to see that for $\lambda \le 3$, \textsc{Radio 2-Coloring} is polynomial-time solvable. The dichotomy for $k=2$ is settled by the fact that \textsc{Radio 2-Coloring} is NP-complete for $\lambda \ge 4$ ~\cite{fiala2001fixed}. Moreover, it can also be verified that any graph $G$ for which $rc_3(G) = 5$, is a disjoint union of paths of length at most $3$. This implies that the question in \textsc{Radio 3-Coloring} can be answered by simply checking (in polynomial time) whether or not $G$ is a disjoint union of paths of length at most $3$. This immediately postulates the conjecture that \textsc{Radio 3-Coloring} is also NP-complete for $\lambda \ge 6$. In fact, noticing that $rn(P_2) = 3$ and $rn(P_3)=5$, we state the following dichotomy conjecture to be proved or disproved.

\begin{question}
    Is \textsc{Radio $k$-Coloring} polynomial-time solvable for $\lambda \le rn(P_k)$ and NP-complete for $\lambda > rn(P_k)$?
\end{question}

\item Another natural question that can be asked is to explore the complexity status of the problem when we vary $k$.

\begin{question}
    For what values of $\lambda$, the \textsc{Radio $k$-coloring} is NP-complete but the \textsc{Radio $(k+1)$-coloring} is polynomial time solvable?
\end{question}

\item It will also be interesting to know the complexity status of the \textsc{Radio $k$-Coloring} problem, where $k\ge 3$, for popularly studied graph classes, such as, 
planar graphs, chordal graphs, bipartite graphs, etc.
\end{enumerate}

\bigskip

\noindent \textbf{Acknowledgements:}
This work is partially supported  by the following projects: ``MA/\\IFCAM/18/39'', ``SRG/2020/001575'', ``MTR/2021/000858'', and ``NBHM/RP-8 (2020)/\\Fresh''. Research by the first author is supported by the French government IDEX-ISITE initiative CAP 20-25 (ANR-16-IDEX-0001), the International Research Center ``Innovation Transportation and Production Systems" of the I-SITE CAP 20-25, and the ANR project GRALMECO (ANR-21-CE48-0004).

\bibliographystyle{abbrv}
\bibliography{bibliography}

\end{document}